\documentclass[a4paper, 11pt,reqno]{amsart}
\usepackage{amsthm,amssymb,amsmath,graphicx,psfrag,enumitem,mathrsfs}
\usepackage{mathtools}
\usepackage[foot]{amsaddr}
\usepackage[british]{babel}
\usepackage[babel]{microtype}
\usepackage[margin=1in]{geometry}

\usepackage[pagewise]{lineno}
\usepackage[textsize=footnotesize]{todonotes}

\usepackage{algorithm}
\usepackage{algpseudocode}
\usepackage{tikz}
\usepackage{setspace}

\usetikzlibrary{backgrounds}

\usepackage[numbers]{natbib}

\makeatletter
\def\NAT@spacechar{~}% NEW
\makeatother

\usepackage{hyperref}
\usepackage{thm-restate}
\usepackage[capitalise]{cleveref}
\hypersetup{colorlinks=true,
   citecolor=blue,
   filecolor=blue,
   linkcolor=blue,
   urlcolor=blue
}

\crefname{figure}{Figure}{Figures}
\Crefname{figure}{Figure}{Figures}

\allowdisplaybreaks

\newtheorem{definition}{Definition}[section]
\newtheorem{claim}{Claim}

\newtheorem{theorem}[definition]{Theorem}

\newtheorem{lemma}[definition]{Lemma}

\newenvironment{claimproof}{%
\let\origqed=\qedsymbol%
\renewcommand{\qedsymbol}{$\blacktriangleleft$}%
\begin{proof}}{\end{proof}\let\qedsymbol=\origqed}
\numberwithin{equation}{section}

\numberwithin{equation}{section}

\newcommand{\comment}[1]{}

\newcommand{\cA}{\mathcal{A}}
\newcommand{\cB}{\mathcal{B}}
\newcommand{\cC}{\mathcal{C}}

\newcommand{\cF}{\mathcal{F}}

\renewcommand{\epsilon}{\varepsilon}

\newcommand{\COMMENT}[1]{}
\renewcommand{\COMMENT}[1]{\footnote{\textcolor{blue!70!black}{#1}}} % comment out to hide comments

\title{A Canonical Polynomial Van Der Waerden's Theorem}

\author[A.~Gir\~{a}o]{Ant\'onio Gir\~{a}o}
\email{giraoa@bham.ac.uk}
\address{School of Mathematics, University of Birmingham, 
Edgbaston, Birmingham, B15 2TT, United Kingdom.}

\thanks{The research leading to these results was partially supported by the EPSRC, grant no. EP/N019504/1.}

\date{\today}

\begin{document}
\onehalfspacing
\maketitle
\begin{abstract}
    We prove a canonical polynomial Van der Waerden's Theorem. More precisely, we show the following. Let $\{p_1(x),\ldots,p_k(x)\}$ be a set of polynomials such that $p_i(x)\in \mathbb{Z}[x]$ and $p_i(0)=0$, for every $i\in \{1,\ldots,k\}$. Then, in any colouring of $\mathbb{Z}$, there exist $a,d\in \mathbb{Z}$ such that $\{a+p_1(d),\ldots,a+p_{k}(d)\}$ forms either a monochromatic or a rainbow set. 
\end{abstract}

\section{Introduction}
Arithmetic Ramsey theory is a branch of combinatorics where one is interested in studying the existence of monochromatic structures in any finite colouring of the integers. A well known theorem in the area due to Van der Waerden \cite{VanDerWaerden} and dating to $1927$ states that in any finite colouring of the natural numbers there exist arbitrarily long monochromatic arithmetic progressions. This theorem has been considerably extended over the years and we emphasize some important extensions. 

A classical result of Rado \cite{Rado} characterizes all integer valued matrix $M$ with the property that in any finite colouring of the naturals there exists a monochromatic solution to $M\cdot \vec{x}=0$. Observe that a solution to the system of linear equations consisting of $x_1-2x_2+x_3=0$, $x_2-2x_3+x_4=0, \ldots, x_{k-2}- 2x_{k-1}+ x_k=0$ forms a $k$-term arithmetic progression. Since such a system is easily seen to satisfy Rado's characterization, Van der Waerden's Theorem follows as a special case of Rado's result. 
Another nice generalisation is the Gallai-Witt's Theorem (see \cite{Grahamrothschildspencer},\cite{Witt}) which states that for any finite subset $A\subset \mathbb{Z}^{n}$, any finite colouring of $\mathbb{Z}^{n}$ contains a monochromatic homothetic copy of $A$. This theorem can be viewed as a multidimensional generalisation of Van der Waerden's Theorem. 

Most ramsey-theoretical results (finite colourings) have a canonical version. In this setting, the palette of colours may be infinite but one still would like to characterize all unavoidable sub-structures. For example, the canonical Van der Waerden's Theorem, first proved by Erd\H{o}s and Graham \cite{ErdosGraham}, says the following.

\begin{theorem}[Canonical Van der Waerden]
Whenever $\mathbb{N}$ is coloured, possibly with infinitely many colours, there exist either arbitrarily long monochromatic arithmetic progressions or arbitrary long rainbow arithmetic progressions.
\end{theorem}

Note that both the Gallai-Witt and Rado's classical theorems have canonical versions. Indeed, the canonical Gallai-Witt's Theorem was originally proved by Deuber, Graham, Pr\"omel and Voigt \cite{Deubergrahampromvoigt} and it was later slightly simplified by Pr\"omel and R\"odl \cite{ProRodl}. Rado's canonical version was proved by Lefmann \cite{Lefmann}.

A more recent and remarkable theorem in arithmetic Ramsey theory which once again extends Van der Waerden's Theorem is the polynomial Van der Waerden's Theorem, originally proved by Bergelson and Liebman \cite{BergLeib}. Their proof uses heavy ergodic theory machinery, however, few years later, a beautiful and purely combinatorial proof was found by Walters \cite{walters}. 

\begin{theorem}[Polynomial Van der Waerden]\label{thm: polyVan}
Let $p_1,p_2,\ldots, p_{k} \in \mathbb{Z}[x]$, where for every $i\in [k]$, $p_{i}(0)=0$ and let $n\in \mathbb{Z}$. Then, there exists $N'\in \mathbb{N}$ such that for every colouring of $\{1,\ldots, N'\}$ with $n$ colours there exist $a,d \in \mathbb{Z}$ such that $\{a,a+p_1(d),\ldots, a+p_{k}(d)\}\subset [N']$ forms a monochromatic set. 
\end{theorem}

The main purpose of this paper is to give a proof of the following canonical version of the polynomial Van der Waerden's Theorem. We remark that our methods might be useful to show canonical versions of other theorems as well as giving new and shorter proofs of known canonical theorems.

\begin{theorem}[Canonical polynomial Van der Waerden]\label{thm: main}
Let $\cA=\{p_1,\ldots, p_{k}\}$, where $p_i \in \mathbb{Z}[x]$ and $p_{i}(0)=0$, for every $i\in\{1,\ldots, k\}$. Then, for any colouring of $\mathbb{Z}$, there exist $a,d \in \mathbb{Z}$ such that $\{a,a+p_1(d),\ldots, a+p_{k}(d)\}$ either forms a monochromatic set, or $\{a,a+p_1(d),\ldots, a+p_{k}(d)\}$ forms a rainbow set. 
\end{theorem}
We remark that our proof of \Cref{thm: main} uses some nice ideas introduced by Walters in \cite{walters}. 

\section{Preliminary definitions and notation}\label{sec: not}
For technical reasons, we will always be considering multi-sets but we will still call them sets. As usual, we denote by $[N]\coloneqq \{1, \ldots, N\}$. We define an \textit{integral} polynomial to be a polynomial with integer coefficients taking the value zero at zero. 
Given a natural number $m$, we say $\Delta: [N] \rightarrow \mathbb{N}^{m}$ is an $m$-type colouring of $[N]$, i.e. a colouring of $[N]$ where each colour $c$ is an element of $\mathbb{N}^{m}$. For an element $a\in [N]$ and $j\in [m]$, we define $\Delta_j(a)$ to be the $j$-th coordinate of $\Delta(a)$. 

Now, let $n,m\in \mathbb{N}$. We say $\Delta: [N] \rightarrow \mathbb{N}^{m}\times \{1,\dots, n\}$ is an $(m,n)$-type colouring of $[N]$. For every set of \textit{distinct} integers $\{a_1,a_2,\ldots, a_k\} \subseteq [N]$, we say it forms a \textit{fully-rainbow} set if the following holds.

\begin{enumerate}[label=$(\mathrm{R\arabic*})$]
    \item\label{def:rainbow} for every $i,i'\in \{1,\ldots, k\}$ and $j,j'\in \{1,\ldots,m\}$, $\Delta_j(a_i)\neq \Delta_{j'}(a_{i'})$,
    \item\label{def:fullyrainbow} there exists $c\in \{1,\ldots, n\}$, such that for every $i \in \{1,\ldots,k\}$, $\Delta_{m+1}(a_i)=c$. 
\end{enumerate} 

We say $\{a_1,\ldots, a_k\}$ (not necessarily distinct) forms a \textit{rainbow} set if it satisfies \ref{def:rainbow}. Finally, we say a set of integers $\{a'_1,\ldots,a'_k\}$ (not necessarily distinct) forms a \textit{monochromatic} set if the following holds.
\begin{enumerate}[label=$(\mathrm{M\arabic*})$]
    \item \label{def:mono} there exists a coordinate $j \in [m]$ such that $\Delta_{j}(a_i)=\Delta_{j}(a_{i'})$, for every $i,i'\in \{1,2,\ldots, k\}$.
\end{enumerate}
Let $\Delta$ be an $(m,n)$-type colouring of $\mathbb{Z}$, and $\cB=\{p'_1,\ldots, p'_{k'}\}$ a set of polynomials. We say that a set of \textit{distinct} integers $A(d)\coloneqq \{a_1,\ldots,a_{k'}\}\subset \mathbb{Z}$ is $\cB$-\textit{focused} at $a\in \mathbb{Z}$ if $a_j-a=p'_j(d)$, for all $j\in \{1,\ldots, k'\}$ and $a\notin \{a_1,\ldots, a_k\}$. Moreover, we say that the sets $A_1(d_1),\ldots, A_q(d_q)$ are \textit{fully-rainbow} $\cB$-focused at $a$, if the following are satisfied.
\begin{enumerate}[label=$(\mathrm{FR\arabic*})$]
    \item \label{def:focus} for all $i\in\{1,\ldots, q\}$, $A_i(d_i)$ is $\cB$-focused at $a$,
    \item \label{def:fr} $A_i(d_i)$ is fully-rainbow,
    \item \label{def: rain} $(\cup_{i=1}^{q}A_i(d_i))$ forms a rainbow set all of whose elements are \textit{distinct}. 
\end{enumerate}
Finally, let $\cF \coloneqq \{A_1(d_1),\ldots,A_q(d_r)\}$ be a collection of fully-rainbow sets $\cB$-focused at $a$. For each $c\in \{1,\ldots,n\}$, let $w_c(\cF)\coloneqq |\{i\mid \Delta_{m+1}(A_i(d_i))=c\}|$. In other words, $w_c(\cF)$ (or $w_c$ whenever $\cF$ is understood from the context) counts the number of fully-rainbow sets $A_i(d_i)\in \cF$ for which $\Delta_{m+1}(A_i(d_i))=c$. For technical reasons, we need to define $g(\cF)\coloneqq \{c\in \{1,\ldots,n\} \mid w_c(\cF)\leq m+1\}$. Now, we let $\|\cF\|\coloneqq \sum_{c\in g(\cF)} w_c$. This is basically an $\ell_1$-norm with a tweak. 

Given $N'\in \mathbb{N}$, we may define the equivalence classes induced by $\Delta$ on intervals of order $N$. Suppose we partition $[N']=I_1\cup \ldots \cup I_{\ell}$ into consecutive intervals of order $N$. Then, we say $I_i\sim_{\Delta} I_j$ if the following hold. We may assume $I_i=I_1=[N]$ and $I_j=[tN]$, for some $t\in \mathbb{N}$.
\begin{enumerate}
    \item for all $i\in \{1,\ldots N\}$, $\Delta_{m+1}(i)=\Delta_{m+1}(tN+i)$,
    \item for all $i,j\in \{1,\ldots, N\}$ and $k\in \{1,\ldots, m\}$, $\Delta_{k}(i)=\Delta_{k}(j)$ if and only if $\Delta_{k}(tN+i)=\Delta_{k}(tN+j)$.
\end{enumerate}
It is easy to see this is indeed an equivalence relation. We denote the set of equivalence classes by $\mathfrak{E}_{N}^{\Delta, N'}$. Crucially, note that for any $m,n, N\in \mathbb{N}$, and any $(m,n)$-type colouring $\Delta$, $\mathfrak{E}_{N}^{\Delta, N'}$ is a finite set and we denote by $f(m,n,N)$, the total number of possible distinct equivalence classes. Also, for an interval $I$, we let $\mathfrak{E}_{N}^{\Delta}(I)$ to be the equivalence class containing $I$. When $N'$ is clear from the context, we omit the superscript $N'$ in the above definitions. 

Following \cite{walters}, let $A=\{p_1,\ldots, p_k\}$ be a set of integral polynomials. Let $D$ be the maximum degree of these polynomials. For $1 \leq i \leq D$, let $N_i(A)$ be the number of distinct leading coefficients of the polynomials in $A$ of degree $i$. We define the \textit{weight vector} $\omega(A)\coloneqq (N_1,\ldots, N_D)$. For any two sets of integral polynomial $A,A'$ we say that $\omega(A) <\omega(A')$ if there exists $r$ such that $N_r(A) <N_r(A')$ and $N_i(A)= N_i(A')$, for every $i> r$. This is easily seen to be a well ordering on the set consisting of all finite sets of integral polynomials. In our proof of \Cref{thm: main}, the `outer' induction will be on the weight vector of $B$.

First, we shall sketch a short proof of the Canonical Van der Waerden's Theorem which makes use of some definitions introduced before.

\section{Proof of the canonical Van Der Waerden's Theorem}

In this section, we give a sketch of a short proof of the Canonical Van der Waerden's Theorem. We hope this will help the reader getting used to some of the terminology and ideas when reading the proof of our main theorem. Our aim in this section is to prove the following result. 
\begin{theorem}

Let $k,t,m\in \mathbb{N}$. Then, there exists $N_0\coloneqq N(k,t,m)\in \mathbb{N}$ such that for every $m$-type colouring $\Delta: [N(k,t,m)] \rightarrow \mathbb{N}^{m}$ one of the two must hold. 
\begin{itemize}
    \item There exists a monochromatic arithmetic progression $A\subseteq [N_0]$ of length $k$ or
    \item there exists a rainbow arithmetic progression $B\subseteq [N_0]$ of length $t$.
\end{itemize}
\end{theorem}

Note that trivially this implies the canonical Van der Waerden's Theorem taking $m=1$ and $k$ arbitrarily large. 

\begin{proof}
The proof goes by induction on $t$. Clearly, $N(k,1,m)$ exists for every $k,m\in \mathbb{N}$.

Suppose we want prove the the existence of $N(k,t+1,m)$. We shall assume by the induction step that $N(k,t,m')$ exists for every $k,m'\in \mathbb{N}$.

 Let $N\in \mathbb{N}$ be a sufficiently large positive integer and let $\Delta$ be a $m$-type colouring of $[2N]$. 

First, we partition $[N]$ into consecutive intervals $I_1,I_2,\ldots, I_{N''}$ each of length $N'$, for some $N''\coloneqq N/N'\in \mathbb{
N}$. The colouring $\Delta$ induces a $(mN')$-type colouring $\Delta'$ of $[N'']$, where the colour of an interval $I_j$ is the vector formed by the concatenation of the colours of the elements of $I_j$ in increasing order. Formally, for every $j\in \{1,\ldots, N''\}$, $\Delta'(I_j)=(\Delta((j-1)N'+1),\Delta((j-1)N'+2),\ldots,\Delta(jN'))$. 

We may assume $N''\geq N(k,t,mN')$. By induction, suppose there exist a coordinate $i\in \{1,\ldots,mN'\}$ and an arithmetic progression $A=\{a_1,a_2,\ldots,a_k\}$ of length $k$ such that $\Delta'_i(a_j)=c$, for every $j\in \{1,\ldots,k\}$. Let $i=mi'+f$, for some $i'\in \{0,\ldots, N'-1\}$ and $f\in \{1,\ldots, m\}$. Then, it follows by construction of $\Delta'$, that $A'=\{a_1N'+mi', a_2N'+ mi', \ldots, a_kN'+mi'\} \subseteq [N]$ is an arithmetic progression and $\Delta_{f}(x)=c$, for every $x\in A'$. Hence, $A'$ forms a monochromatic progression, as we wanted to show. 
We may then assume $[N'']$ contains a rainbow arithmetic progression of length $t$. Let $A^{*}=\{a'_{1},a'_2,\ldots, a'_t\}\subseteq [N'']$ be such a rainbow arithmetic progression and let $d>0$ be the progression difference.

Now, let $a_{t+1}\coloneqq a_{t}+d\in [2N'']$ and $I_{a_{t+1}}$ the corresponding interval. Also, let $x\in [2N]$ be the largest element of $I_{a_{t+1}}$ and $\Delta(x)=(c_1,\ldots, c_m)$. 
Observe that for any $q\in \{0,\ldots, N'-1\}$, the sets $T_q \coloneqq \{x,x-(dN'+q),x-2(dN'+q),\dots,x-t(dN'+q)\}\subseteq [2N]$ form an arithmetic progression of length $t+1$. Moreover, since $x-j(dN'+q)\in I_{a_{t+1-j}}$ and $A^{*}$ is rainbow, we have that every $T_q\setminus\{x\}$ forms a rainbow set.

Let us look at $I_{a_t}$. We will construct now a finite colouring $s_t: I_{a_t}\rightarrow \{0\}\cup [m]\times [m]$. For $\ell \in I_{a_t}$, $s_t(\ell)\coloneqq (i,j)\in [m]\times [m]$, if there exists $i,j\in \{1,\ldots, m\}$ such that $\Delta_i(\ell)=c_j$ (if there are many such pairs, choose one arbitrarily). If no such $i,j$ exist, then we set set $s(\ell)=0$. This is a finite colouring so by Van der Waerden's Theorem either there exists a $k$-term arithmetic progression $A_{t}\subset I_{a_t}$ of colour $(i,c_j)\in [m]\times [m]$, in which case we are done, because $A''$ forms a monochromatic arithmetic progression(see \ref{def:mono}) or there exists a sufficiently large arithmetic progression $P_t\subset I_{a_t}$ of colour $0$. Observe that by construction $\{y,x\}$ forms a rainbow set (see \ref{def:rainbow}), for every $y\in P_t$. Let $P_t\coloneqq\{x-a_t,x-a_t-d_t),\ldots, x-a_t-p_td_t)\}$, for some $a_t,d_t\in [N']$ and a sufficiently large $p_t\in \mathbb{N}$.
Let $I_{t-1}\coloneqq\{x-2(dN'+a_t), x-2(dN'+a_t+d_t),\ldots, x-2(dN'+a+d_tp_t)\}$ be arithmetic progression of length $p_t$ inside $I_{a_{t-1}}$. We will apply the same reasoning to $I_{t-1}$ as we did with $I_t$. We define a finite colouring $s_{t-2}s_t: I_{t-1}\rightarrow \{0\}\cup [m]\times [m]$, as before. By Van der Waerden's Theorem, either there exists a $k$-term arithmetic progression $A_{t-1}\subset I_{t-1}$ of colour $(i,c_j)\in [m]\times [m]$, in which case we are done, or there exists a sufficiently large arithmetic progression $P_{t-1}\subset I_{t-1}$ of colour $0$. Let $P_{t-1}\coloneqq\{x-a_{t-1},x-a_{t-1}-d_{t-1},\ldots, x-a_{t-1}-p_{t-1}d_{t-1}\}$, for some $a_{t,-1},d_{t-1}\in [N']$ and sufficiently large $p_{t-1}\in\mathbb{N}$.
We may continue in the same fashion for $t$ all the way down to $1$. Indeed, suppose we have constructed $P_t\subseteq I_{a_t},P_{t-1}\subset I_{t-1}\subset I_{a_{t-1}}\ldots, P_{t-i}\subset I_{t-i}\subseteq I_{a_{t-i}}$ and we wish to construct $P_{t-i-1}\subseteq I_{t-i-i}\subseteq I_{a_{t-i-1}}$. Suppose $P_{t-i}\coloneqq \{ x- a_{t-i},\ldots, x-a_{t-i}-d_{t-i}p_{t-i}\}$, for some $a_{t-i},d_{t-i}\in [N']$ and sufficiently large $p_{t-i}\in \mathbb{N}$. Then, let $I_{t-i-1}\coloneqq\{x-2a_{t-i}-dN',x-2(a_{t-i}+d_{t-i})-dN',\ldots, x-2(a_{t-i}+d_{t-i}p_{t-i})-dN'\}\subseteq I_{a_{t-i-1}}$. As before, we construct a finite colouring $s_{t-i-1}: I_{t-i-1}\rightarrow \{0\}\cup [m]\times [m]$. Then, we can either find a $k$-term arithmetic progression $A_{t-i-1}\subset I_{t-i-1}$ of colour $(i,c_j)\in [m]\times [m]$, in which case we are done, or we can find a sufficiently large arithmetic progression $P_{t-i-1}\subset I_{t-i-1}$ of colour $0$. 
Note now, that as long as $P_{1}\neq \varnothing$ (which is guaranteed by starting with a sufficiently large $p_t$) we may find a $t$-term rainbow arithmetic progression. Indeed, let $x-a_1\in P_1$ be the largest element in $P_1$. Note that by construction $x-a_1/2^{i-1}+(i-1)dN'\in I_{i}$, for every $i\in \{1,\ldots, t\}$. Therefore, the set $\{x\}\cup \{x-a_1/2^{i-1}+(i-1)dN'\mid i\in \{1,\ldots, t\}\}$ forms a rainbow arithmetic progression of length $t+1$, as we wanted to show. 
 \end{proof}

\section{Proof of Theorem~\ref{thm: main}}

First, we need to show a simple lemma concerning integer valued polynomials. 

\begin{lemma}\label{lem1}
Let $\cA=\{p_1,\ldots, p_k\}$ be a collection of distinct integral polynomials. Then, there exists $h\in \mathbb{N}$ such that for every $i,j\in \{1,\ldots, k\}$ (possibly $i=j$), and for every $h'>h$, $p_j(x)\neq p_{i}(h'+x)-p_{i}(h')$ as elements of $\mathbb{Z}[x]$.
\end{lemma}
\begin{proof}
If not, there exist $i,j\in \{1,\ldots,k\}$, and $h_1\in \mathbb{N}$ such that $p_i(x)=p_{j}(h_1+x)-p_{i}(h_1)$. Let $p_i(x)=a_nx^{n}+\ldots+a_1x$ and $p_j(x)=b_{n'}x^{n'}+\ldots+b_{1}x$. By substituting,   $p_i(x)=p_j(h_1+x)-p_j(h_1)=b_{n'}(h_1+x)^{n'}+b_{n'-1}(x+h_1)^{n'-1}+\ldots+b_1(x+h_1)-p_j(h_1)$, hence $n=n'$, $a_n=b_{n}$ and $n\cdot h_1+b_{n-1}=a_{n-1}$. 
Therefore, $h_1=(a_{n-1}-b_{n-1})/n$, which contradicts the fact $h_1$ could have been chosen sufficiently large. 
\end{proof}

%\begin{lemma}\label{lem2} 
%Let $\cA=\{p_1,\ldots, p_k\}$ be a collection of distinct integral polynomials. Then, there exists $h\in \mathbb{N}$, such that for every integer $h'>h$, we have $p_i(h')\neq p_{j}(h')$, for every distinct $i,j\in \{1,\ldots, k\}$. 
%\end{lemma}

%\begin{proof}
%It is enough to show that for any two distinct polynomials $p(x)$ and $q(x)$, there exists $h$ such that for every $h'>h$,
%$p(h')\neq q(h')$. Let $p(x)=a_nx^{n}+\ldots+a_1x$ and $q(x)=b_{n'}x^{n'}+\ldots+b_1x$, where $a_n,b_{n'}\neq 0$.
%Note that if $n>n'$ (or $n<n'$) then $|p(x)|$ is eventually always larger than $|q(x)|$ (or vice-versa). Now, let $r=\max\{t: a_{t}\neq b_{t}\}$. Observe $1\leq r\leq n$ since $p(x)\neq q(x)$. 
%We may assume that $a_{t}>b_{t}$. Then, $|p(x)|$ is eventually always larger than $|q(x)|$, as we wanted to show. 
%\end{proof}
Given a collection $\cA=\{p_1,\ldots, p_k\}$ of distinct integral polynomials, we define $h(\cA)$ to be the smallest positive integer for which $p_i(x)\neq p_j(x+h')-p_j(h')$, for every $h'>h(\cA)$ and every $i\neq j\in \{1,\ldots, q\}$. This is well defined by \Cref{lem1}.

In order to prove Theorem~\ref{thm: main}, it will be useful to prove the following slightly stronger statement, from which \cref{thm: main} can be easily deduced. 

\begin{theorem}
Let $h,k,k',m,n\in \mathbb{N}$, $\cA=\{p_1,\ldots, p_{k}\}, \cB=\{p'_{1},\ldots, p'_{k'}\}$ be two sets of integral polynomials. Moreover, suppose that for every $i,j\in \{1,\ldots,k'\}$, $p'_i\neq p'_j$. Then, there exists $N\in \mathbb{N}$ such that for every $(m,n)$-type colouring of $[N]$, one of the following holds.
\begin{itemize}
    \item there exist $a,d \in \mathbb{Z}$ such that $\{a,a+p_1(d),\ldots, a+p_{k}(d)\}\subseteq [N]$ forms a monochromatic set,
    \item there exist $a',d' \in \mathbb{Z}$ such that $d'>h$ and $\{a',a'+p'_1(d'),\ldots, a'+p'_{k}(d')\}\subseteq [N]$ forms a fully-rainbow set.  
\end{itemize}
\end{theorem}

\begin{proof}[Proof of Theorem~\ref{thm: main}]

The induction will be on the weight of $\cB$. We may and will assume $h>h(\cB)$. \\

\textbf{Outer induction hypothesis.} For every $h,m,n,k,k'\in \mathbb{N}$, any two sets of integral polynomials $\cA=\{p_1,\ldots,p_k\}$ and $\cB=\{p'_1,\ldots, p'_{k'}\}$ where $p'_i\neq p'_j$ (for all $i,j\in \{1,\ldots, k'\}$), there exists $[N]$, such that for any $(m,n)$-type colouring of $[N]$, there exist $a,d\in \mathbb{Z}$ such that $\{a,a+p_1(d),\ldots, a+p_{k}(d)\}\subseteq [N]$ forms a monochromatic set, or there exists $a,d'>h$ such that $\{a,a+p'_1(d),\ldots, a+p'_{k}(d)\}\subseteq [N]$ forms a fully-rainbow set. \\

Suppose that the outer induction hypothesis is true for all $h,m,n,k\in \mathbb{N}$, for every set of integral polynomials $\cA$ of order $k$ and for every set $\cB'$ of distinct integral polynomials satisfying $\omega(B')<\omega(B)$. 

To check the base case, let $h,m,n\in \mathbb{N}$ and $B'=\{a_1x\}$, for some $a_1\in \mathbb{Z}\setminus\{0\}$. 
Let $N'\in \mathbb{N}$ be given by \Cref{thm: polyVan} when applied to the collection of integral polynomials $p_1^{*}\coloneqq p_1(a_1hx)/(a_1h),\ldots, p_k^{*}(x)\coloneqq p_k(a_1hx)/(a_1h)$ and $n'\coloneqq ((m+1)^{2}\cdot n)^{n}$ playing the role of $n$.

Let $\Delta$ be an $(m,n)$-type colouring of $[a_1hnN']$, our aim is to prove we can find either a rainbow set $\{a,a+a_1(hd)\}\subset [a_1hnN']$ or a monochromatic set $\{a,a+p_1(d),\ldots, a+p_k(d)\}\subset [a_1hnN']$, for some $d\in \mathbb{N}$. 
First, let $\{a_1hx_1,\ldots,a_1hx_t\}\subseteq [a_1hnN']$ be a largest set such that $\Delta_{m+1}(a_1hx_i)\neq \Delta_{m+1}(a_1hx_{i'})$, for every $i,i'\in \{1,\ldots, t\}$. Note that $t\leq n$ and therefore there must exist an interval $I\subseteq [a_1hnN']$, where $|I|\geq a_1hN'$ and $x_i\notin I$, for every $i\in \{1,\ldots,t\}$.  
Clearly, we may assume $I=[a_1hN']$ since intervals are translation invariant with respect to satisfying our main theorem. 

Now, consider the following $n'$-colouring of $[N']$, $\Delta^{*}:[N']\rightarrow \{1,\ldots, n'\}$, where $\Delta^{*}(x)=c$, for some $c\in \{(i,j,b)\mid i,j\in \{1,\ldots, m\}, b\in \{1,\ldots,n\}\}^{t}$. For $g\in \{1,\ldots, t\}$, $\Delta^{*}_g(x)=(i_1,j_2,b_3)$, where $i_1,j_2\in \{1,\ldots,m\}$ are any two indices for which $\Delta_{i_1}(a_1hx)=\Delta_{j_2}(a_1hx_g)$, if no such two indices indices exist then $i_1=j_2=m+1$. Finally, we let $b_3=\Delta_{m+1}(a_1hx_g)$, i.e. the third coordinate of $\Delta_{g}^{*}(x)$ equals the last coordinate of $\Delta(a_1hx_g)$. Clearly, $\Delta^{*}$ is an $n'$-colouring of $[N']$ and hence by construction there exists a monochromatic set $\{a,a+p_1^{*}(d),\ldots, a+p_k^{*}(d))\}\subseteq [N']$ of colour $c=\prod_{r=1}^{t}(i_r,j_r,b_r)$, where $i_r,j_r\in \{1,\ldots,m+1\}$ and $b_r\in \{1,\ldots, n\}$. Suppose there exists $r\in \{1,\ldots, t\}$ for which $i_r,j_r\in \{1,\ldots,m\}$, then $M\coloneqq \{a_1ha,a_1ha+p_1(a_1hd),\ldots,a_1ha+p_{k}(a_1hd)\}\subset [a_1hN']$ forms a monochromatic set with respect to $\Delta$. Indeed, observe that $\Delta_{i_r}(y)=\Delta_{j_r}(a_1hx_{r})$, for all $y\in M$. Suppose, on the other hand, $i_r=j_r=m+1$, for all $r\in \{1,\ldots, t\}$, and $\Delta_{m+1}(a_1ha)=\Delta_{m+1}(a_1hx_q)$, for some $q\in \{1,\ldots t\}$. Hence, $\{a_1ha,a_1ha+a_1(h(x_q-a))=a_1hx_q\}\subseteq [a_1hnN']$ forms a rainbow set, as we wanted to show.  \\

\textbf{Inner induction hypothesis.}
For all $r\leq (m+1)n$ there exist $N\in \mathbb{N}$ such that if $\Delta$ is an $(m,n)$-type colouring of $[N]$, then at least one of the following holds.

\begin{enumerate}[label=$(\mathrm{\roman*})$]
    \item \label{itm:2fullyrainbow} there exist a collection $\cF=\{A_1(d_1),\ldots, A_{q}(d_{q})\subseteq [N]\}$ of fully-rainbow sets $\cB$-focused at $a$, for some $a\in \mathbb{Z}$, such that $d_i>h$, for all $i \in \{1,\ldots, q\}$. 
    Furthermore, $\|\cF\|=r$. 
    \item \label{itm:2rainbow} there exists $a',d'\in \mathbb{Z}$ such that $d>h$ and the set $\{a',a'+p'_1(d'),a+p'_2(d'),\ldots, a+p'_{k'}(d')\}$ is fully-rainbow,
    \item \label{itm:2mono} there exist $a,d\in \mathbb{Z}$ such that the set $\{a,a+p_1(d),a+p_2(d),\ldots, a+p_k(d)\}$ is monochromatic.
\end{enumerate}

From this hypothesis, we prove our result by setting $r=q(m+1)$. To see this, note that if either \ref{itm:2rainbow}, or \ref{itm:2mono} hold, we are done. 
On the other hand, suppose \ref{itm:2fullyrainbow} holds and let $\cF$ be such a collection with $\|\cF\|=(m+1)n$. Let $\Delta_{m+1}(a)=c$, for some $c\in \{1,\ldots,n\}$. Observe that by assumption on the norm of $\cF$, there are $m+1$ sets $A_{i_1}(d_{i_1}),\ldots,A_{i_{m+1}}(d_{m+1})\in \cF$ such that $\Delta_{m+1}(A_{i_j}(d_{i_j})=c$, for all $j\in \{1,\ldots,m\}$. 
Now, we show at least one of the $A_{i_j}(d_{i_j})$'s has the property that $A_{i_j}(d_{i_j})\cup\{a\}$ forms a rainbow set and hence a fully-rainbow set, as required. Suppose not, then for all $w\in \{1,\ldots,m+1\}$, there are $i(w),i'(w)\in \{1,\ldots, m\}$ and $x(w)\in A_{i_w}(d_{i_w})$ such that $\Delta_{i(w)}(a)=\Delta_{i'(w)}(x(w))$, hence there must exist $w,w'\in \{1,\ldots, m+1\}$ where $i(w)=i(w')$, which contradicts the fact $A_{i_w}(d_{i_w}) \cup A_{i_{w'}}(d_{i_{w'}})$ forms a rainbow set. Therefore, if \ref{itm:2fullyrainbow} holds for $r=n(m+1)$, then there exist $a,d\in \mathbb{Z}$, where $d>h$ such that the set $\{a,a+p'_1(d),a+p'_2(d),\ldots, a+p'_k(d)\}\subseteq [N]$ is fully-rainbow, as we wanted to show. 

Now, we turn to the proof the inner induction hypothesis. The induction will be on $r$. Suppose the first inner induction hypothesis is true for $r-1$ taking $N\in \mathbb{N}$. We will show that there is $N'\in \mathbb{N}$ satisfying the hypothesis for $r$ (an upper bound for $N'$ could be computed but for simplicity of the argument we will avoid doing this). Throughout the proof, we will assume that neither \ref{itm:2rainbow} or \ref{itm:2mono} hold. As in \cite{walters}, let $d_{max}$ be the largest $d>h$ for which there exist $a,a_1,\ldots,a_k\subset [N]$ satisfying $a_i-a=p'_i(d)$, for all $i\in\{1,\ldots, k'\}$. Note $d_{max}$ exists since all polynomials in $\cB$ tend to infinity. We may assume that $p'_1$ has minimal degree amongst the polynomials in $\cB$. We now define the set $ \cB^{*}$ consisting of the following polynomials
\begin{align*}
     &p'_{d_i,j}(x)\coloneqq p'_j(x+d_i)-p'_1(x)-p'_j(d_i) \quad h <d_i \leq d_{max}, \, 1\leq i\leq k', \text{ and }\\
     & p'_{0,j}(x)\coloneqq p'_j(x)-p'_1(x) \quad 1\leq j\leq k'.
\end{align*}

By taking a subset $\cB^{*}$, we may assume all polynomials are distinct. Clearly, these polynomials are integral. More importantly, $\omega(\cB^{*})<\omega(\cB)$. To see this, suppose that $p'_j$ has larger degree than $p'_1$. Then, all polynomials $p'_{d_i,j}$, for $h< d_i\leq d_{max}$ or $d_i=0$ have the same leading coefficient and the same degree as $p'_j$. If $p'_j$ has the same degree but a different leading coefficient from that of $p'_1$, then all polynomials $p'_{d_i,j}$, for $h< d_i\leq d_{max}$ or $d_i=0$ have the same leading coefficient equal to the leading coefficient of $p'_j-p'_1$. Finally, if $p'_j$ has the same degree and leading coefficient as $p'_1$, then all the polynomials $p'_{d_i,j}$, for $h< d_i\leq d_{max}$ or $d_i=0$, have smaller degree than $p'_1$.
This implies that $\omega_r(\cB^{*})=\omega_r(\cB)$, for all $r> deg(p'_1)$ and $\omega_r(\cB^{*})=\omega_r(\cB)-1$, for $r=deg(p'_1)$. (The coordinates of $\omega(\cB^{*})$ may increase for $r<deg(p'_1)$). Thus, $\omega(\cB^{*})<\omega(\cB)$, as we wanted to show. By assumption on $h$, $p'_{0,j}(x)\neq p'_{d_i,j'}(x)$ for every $h<d_i\leq d$ and every $j,j'\in \{1,\ldots,k'\}$.

We will have to modify the polynomials in $\cA$ and $\cB^{*}$ slightly. We need to do this since later in the proof we are going to divide $[N']$ into blocks of length $N$ and we need to take this into account.

Let $q_j(x)\coloneqq p_j(Nx)/N$ and $q'_{d_i,j}(x)\coloneqq p'_{d_i,j}(Nx)/N$, for every $p_j\in \cA$ and $p'_{d_i,j}\in \cB^{*}$. Let $\cA'$ and $\cB'$ be the set consisting of the polynomials $q_j$ and $q'_{d_i,j}$, respectively. It is easy to see that all polynomials in $\cA',\cB'$ are integral polynomials and $\cB'$ still forms a collection of distinct integral polynomials. Also, observe that $\omega(\cB')=\omega(\cB^{*})$ since, although the leading coefficients may change, the number of distinct leading coefficients of polynomials of a given degree does not. Thus the outer induction hypothesis applies to $\cA'$, $\cB'$.
\begin{enumerate}[label=$(\mathrm{P}\arabic*)$]
    \item \label{itm:p} By definition of $h(\cB)$, and the fact $h>h(\cB)$ we have the following. For every $h<d_i\leq d_{max}$ and every $j,j'\in \{1,\ldots,k'\}$, $q'_{0,j}(x)\neq q'_{d_i,j'}(x)$ (as elements of $\mathbb{Z}[x]$).
\end{enumerate}

%Let $h'>0$ be the positive integer given by \Cref{lem2} when applied to the collection $\cB'$.

Now, we divide $[N']$ into intervals of size $N$ and we let $C_s\coloneqq \{N(s-1)+1,\ldots, Ns\}$, for every $s\in \{1,\ldots, N'' \coloneqq N'/N \}$. As seen in Section~\ref{sec: not}, $\Delta$ induces an equivalence relation $\sim_{\Delta}$ on $\{C_1,\ldots,C_{N''}\}$. Since there are at most $f(N,n,m)$ distinct equivalence relations, we may apply the outer induction hypothesis to the sets $\cA'$, $\cB'$, $h$, $N\cdot m$, and $f(N,n,m)$ playing the roles of $\cA$, $\cB$, $h$, $m$, and $n$. For every $s\in [N'']$, let $\Delta'(s)=(\Delta_1(N(s-1)+1),\ldots, \Delta_{m}(N(s-1)+1),\ldots, \Delta_1(Ns),\ldots, \Delta_{m}(Ns),\mathfrak{E}^{\Delta}(C_s))$. By definition, $\Delta'$ is an $(N\cdot m, f(N,n,m))$-type colouring of $[N'']$. Therefore, provided $N''$ is sufficiently large, one of the following holds.\\

\textbf{Case $1$.} There is $s',d'\in \mathbb{Z}$ and a collection of intervals $\cC'\coloneqq\{C'_{s'},C'_{s'_j} \mid 1\leq j\leq k\}$, where $s'_{j}-s'=q_{j}(d')$ and $B\coloneqq\{s',s'_j \mid1\leq j\leq k\}\subseteq [N''] $ forms a monochromatic set with respect to $\Delta'$.\\

\textbf{Case $2$.} There exist $s,d\in \mathbb{Z}$, $d>h$ and a collection of intervals $\cC\coloneqq\{C_s,C_{s_{d_i,j}} \mid h < d_i \leq d_{max}, \text{ or } d_i=0, \text{ and } 1\leq j\leq k'\}$, where $s_{d_i,j}-s=q'_{d_i,j}(d)$ and $A\coloneqq\{s,{s_{d_i,j}} \mid h \leq d_i \leq d_{max}, \text{ or } d_i=0 \text{ and } 1\leq j\leq k'\}\subseteq [N'']$ is fully-rainbow with respect to $\Delta'$. \\

First, let us suppose \textbf{Case $1$.} holds. From the definition of a monochromatic set, we know there exists an index $i(B)\in \{1,\ldots, N\cdot m\}$, such that $\Delta'_{i(B)}(s'_j)=\Delta'_{i(B)}(s')$, for all $j\in \{1,\ldots,k\}$. Let $i(B)=(i'(B)-1)\cdot m+\ell$, for some $1\leq i'(B)\leq N$ and $1\leq \ell \leq m$.

\begin{claim}\label{claim1}
The set $A'\coloneqq\{(s'-1)\cdot N+i'(B),(s'_1-1)\cdot N+i'(B),\ldots, (s'_k-1)\cdot N+i'(B)\}\subseteq [N']$ forms a monochromatic set with respect to $\Delta$. 
\end{claim}
\begin{claimproof}
Observe that for every $j\in \{1,\ldots,k\}$, $((s'_j-1)N+i'(B))-((s'-1)N+i'(B))=N\cdot q_{j}(d')=p_j(Nd')$, as required. Moreover, we have $(s'_j-1)N+i'(B)\in C_{s'_j}$, for every $j\in \{1,\ldots,k\}$. 
By construction of $\Delta'$, we have that $\Delta_{\ell}((s'-1)N+i'(B))=\Delta_{\ell}((s'_j-1)N+i'(B))$, for every $j\in\{1,\ldots, k\}$ and therefore $A'$ forms a monochromatic set. 
\end{claimproof}

This is a contradiction, as we assumed no such monochromatic set exists in $[N']$. Hence, \textbf{Case $2$.} must hold. 
Now, observe that by the choice of $N$ and the assumption that there do not exist $a,d\in \mathbb{Z}$ with $\{a,a+p_1(d),\ldots,a+p_k(d)\} $ forming a monochromatic set or $a'$ and $d''>h$ with $\{a',a'+p'_1(d''),\ldots, a'+p'_{k'}(d'')\}$ forming a fully-rainbow set, it follows that $C_s$ contains a collection $\cF=\{A'_1(d_1), \dots, A'_{q'-1}(d_{q'-1})\subseteq C_s\}$ of fully-rainbow $\cB$-focused sets at $a\in C_s$ such that $h < d_1,\ldots, d_{q'-1}\leq d_{max}$, where $\|\cF\|=r-1$. 
Suppose that for every $i\in \{1,\ldots,q'\}$, $A'_i(d_i)=\{a_{i,j}\mid 1 \leq j \leq k'\}$ and $a_{i,j}-a=p'_{j}(d_i)$. We prove now the following claim.

\begin{claim}
Let $d_0\coloneqq 0$, $A_0(0)\coloneqq \{a\}$ and $a_{0,j}\coloneqq a$, for all $j\in \{1,\ldots,k'\}$. Then, for every $i\in \{0,1,\ldots,q-1\}$, the sets $A'_i(N(d+d_i))\coloneqq\{a_{i,j}+N\cdot q'_{d_i,j}(d)\mid1 \leq j\leq k'\}\subset [N']$ form a collection $\cF'$ of fully-rainbow sets, $\cB$-focused at $a- p_1(Nd)$ and $\|\cF'\|=r$.
\end{claim}
\begin{claimproof}
First, we need to show that for every $i\in \{0,1,\ldots,q-1\}$, $A'_i(N(d+d_i))$ is $\cB$-focused at $a-p_1(N\cdot d')$. 
To see this observe that for every $j\in \{1,\ldots,k'\}$,
\begin{align*}
a_{i,j}+N\cdot q'_{d_i,j}(d)-(a-p'_1(Nd))&=(a_{i,j}-a)+N\cdot q'_{d_i,j}(d)+p'_1(Nd)\\
&=p'_{j}(d_i)+(p'_j(N(d_i+d))-p_j(d_i)-p'_1(Nd))+p'_1(Nd)\\
&=p'_j(N(d_i+d)),
\end{align*}
and \ref{def:focus} is satisfied. We also need to show that every $A'_i(N(d+d_i))$ forms a fully-rainbow set. Note that $a_{i,j}+N\cdot q'_{d_i,j}(d)\in C_{s_{d_i},j}$ and if $C_{s_{d_i},j}\neq  C_{s_{d_{i'}},j'}$, then for any $x\in C_{s_{d_i},j}$ and $y\in C_{s_{d_{i'},{j'}}}$, $\{x,y\}$ is rainbow with respect to $\Delta$.  

Now, it is easy to see that for every $i\in \{0,\ldots, q\}$, $q'_{d_i,j}(x)\neq q'_{d_i,j'}(x)$ if $j\neq j'$, hence by the definition of $A'$, we must have that $C_{s_{d_i},j}\neq C_{s_{d_i},j'}$. Therefore, by the construction of $\Delta'$, for every $i\in \{0,\ldots,q-1\}$, $A'_i(N(d+d_i))$ forms a rainbow set. Finally, since $\mathfrak{E}^{\Delta}(C_s)=\mathfrak{E}^{\Delta}(C_{s_{d_i},j})$, for all $j\in \{1,\ldots, k'\}$, it implies in particular, that $\Delta_{m+1}(a_{i,j})=\Delta_{m+1}(a_{i,j}+N\cdot q'_{d_i,j}(d))$. Hence, for every $i \in \{0,\ldots, q-1\}$, $A'_i(N(d+d_i))$ forms a fully-rainbow set and \ref{def:fr} holds. It remains to show \ref{def: rain}. Clearly, $A'_i(N(d+d_i))\cap A'_{i'}(N(d+d_{i'}))=\varnothing$, for every $i,i' \in \{0,\ldots, q-1\}$. Indeed, this holds because every element of $A'_i(N(d+d_i))$ is a translation of an element of $A'_i(d_i)$ by a multiple of $N$, since by assumption, $A'_i(d_i)\cap A'_{i'}(d_{i'})=\varnothing$ and $A'_i(d_i)\subset C_s$, for all $i\neq i'\in \{1,\ldots, q-1\}$, all elements in $\cup_{i=0}^{q-1}A'_i(N(d+d_i))$ are distinct. 
To conclude the proof of \ref{def: rain}, we just need to show that $\cup_{i=0}^{q-1}A'_i(N(d+d_i))$ forms a rainbow set. Recall that by \ref{itm:p}, $q'_{0,j}(x)\neq q'_{d_i,j'} $ which implies that $C_{s_0,j}\neq C_{d_{i},j'}$, for all $h<d_i\leq d_{max}$ and $j',j\in \{1,\ldots, k'\}$. Hence, by the above, $A'_0(0)\cup A'_i(d_i)$ forms a rainbow set for all $i\in \{1,\ldots,k'\}$. 
Finally, note that since $A'$ is fully-rainbow with respect to $\Delta'$, we have that $\mathfrak{E}_{\Delta}(C_s)=\mathfrak{E}_{\Delta}(C_{s_{d_i,j}})$, for every $h< d_i\leq d_{max}$ and $j\in \{1,\ldots k'\}$. Suppose for contradiction that $\{a_{i,j}+Nq'_{d'_i,j}(d),a_{i',j'}+Nq'_{d_{i'},j'}(d)\}$ is not rainbow with respect to $\Delta$. Then, $i\neq i'$ (since we already have proved $A'_i(N(d+d_i))$ is rainbow), also $i\neq 0$ and $i'\neq 0$ (since we have proved $A'_0(0)\cup A'_i(d_i)$ is rainbow). 
First, if $C_{s_{d_i},j}\neq C_{s_{d_{i'},j}}$, we are done by the above observation. So we must have that $C_{s_{d_i},j}= C_{s_{d_{i'}},j}$ or equivalently $q'_{d'_i,j}(d)=q'_{d_{i'},j'}(d)$. 

In this case, we use the fact that $\mathfrak{E}^{\Delta}(C_s)=\mathfrak{E}^{\Delta}(C_{s_{d_i},j})=\mathfrak{E}^{\Delta}(C_{s_{d_{i'}},j'})$, which implies that $\{a_{i,j},a_{i',j'}\}$ is rainbow if and only if $\{a_{i,j}+Nq'_{d'_i,j}(d),a_{i',j'}+Nq'_{d'_{i'},j'}(d)\}$ is rainbow. Since the former is rainbow with respect to $\Delta$, we obtain the desired contradiction. Let us show now that $\|\cF'\|=r$. 
First, we may assume that $w_{i}(\cF)=m+1$, for all $i\in \{1,\ldots, n_1\}$ and $w_{i}(\cF) <m+1$, for $n_1 <i \leq n$. From the definition of $\|\cF\|$, we have that $\sum_{j=n_1}^{n}w_{j}(\cF)=r-1$.
Let $\Delta_{m+1}(a)=c$ and suppose that $c\in \{1,\ldots,n_1\}$. Let $$\cF_c\coloneqq\{A'_{\ell_1}(d_{\ell_1}),\ldots,A'_{\ell_{m+1}}(d_{\ell_{m+1}})\mid A'_{\ell_i}(d_{\ell_i})\in \cF \text{ and } \Delta_{m+1}(A'_{\ell_i}(d_{\ell_i}))=c\}.$$
Then, there must $j\in \{1,\ldots,m+1\}$ such that $A'_{\ell_j}(d_{\ell_j})\cup \{a\}$ forms a rainbow set and hence a fully-rainbow set, contradicting the fact \ref{itm:2rainbow} does not hold. 
Indeed, if for all $j\in  \{1,\ldots,m+1\}$, there are $i(j),i'(j)\in \{1,\ldots, m\}$ and $x\in A'_{\ell_j}(d_{\ell_j})$ such that $\Delta_{i(j)}(a)=\Delta_{i'(j)}(x)$, there must exist, by pigeon-hole principle, $j\neq f\in \{1,\ldots, m+1\}$ 
for which $A'_{\ell_j}(d_{\ell_j})\cup A'_{\ell_f}(d_{\ell_f})$ is not rainbow, contradicting \ref{def: rain}. 
Therefore, $c\notin \{1,\ldots, n_1\}$. By the above, $\Delta_{m+1}(a)=\Delta_{m+1}(A'_0(0))=c$. Moreover, it is easy to see that $w_j(\cF')=w_j(\cF)$, for all $j\in \{1,\ldots, n\}\setminus \{c\}$ as $\Delta_{m+1}(A'_i(N(d+d_i)))=\Delta_{m+1}(A'_i(d_i))$, for all $i\in \{1,\ldots, q'-1\}$. Hence, $\|\cF'\|=\|\cF\|+1=r$. 
\end{claimproof}

With this claim, we have shown $\cF'=\{A_0(0),A'_1(N(d+d_1)),\ldots,A'_{q'-1}(N(d+d_{q'-1}))\subseteq [N']\}$ contains $q'$ fully-rainbow sets $\cB$-focused at $a-p_1(Nd)$, where $Nd,N(d+d_1),\ldots,N(d+d_{q-1})>h$, and $\|\cF'\|=r$, as required for the inductive step.

This proves the inner induction hypothesis and concludes the proof of \cref{thm: main}. 

\end{proof}

\bibliography{CanPVDW}{}
\bibliographystyle{plain}

\end{document}